\documentclass[11pt]{article}
\usepackage{amsmath}
\usepackage{amsfonts}
\usepackage{amsthm}
\usepackage{amssymb, setspace}
\usepackage{color,graphicx,epsfig,geometry,hyperref,fancyhdr}
\usepackage[T1]{fontenc}
\usepackage{float}
\usepackage{subfig}
\usepackage{commath}
\usepackage{bbm}
\usepackage{listings}
\usepackage{mathrsfs,fleqn}
\newtheorem{theorem}{Theorem}[section]

\newtheorem{lemma}{Lemma}[section]

\newcommand{\N}{\mathbb{N}}

\newcommand{\R}{\mathbb{R}}
\newcommand{\C}{\mathbb{C}}

\newcommand{\grad}{\nabla}

\graphicspath{{paper-pics/}}

\begin{document}

\begin{flushleft}
\Large 
\noindent{\bf \Large Direct sampling method via Landweber iteration for an absorbing scatterer with a conductive boundary}
\end{flushleft}

\vspace{0.2in}

{\bf  \large Rafael Ceja Ayala and Isaac Harris}\\
\indent {\small Department of Mathematics, Purdue University, West Lafayette, IN 47907 }\\
\indent {\small Email:  \texttt{rcejaaya@purdue.edu}  and \texttt{harri814@purdue.edu} }\\

{\bf  \large Andreas Kleefeld}\\
\indent {\small Forschungszentrum J\"{u}lich GmbH, J\"{u}lich Supercomputing Centre, } \\
\indent {\small Wilhelm-Johnen-Stra{\ss}e, 52425 J\"{u}lich, Germany}\\
\indent {\small University of Applied Sciences Aachen, Faculty of Medical Engineering and } \\
\indent {\small Technomathematics, Heinrich-Mu\ss{}mann-Str. 1, 52428 J\"{u}lich, Germany}\\
\indent {\small Email: \texttt{a.kleefeld@fz-juelich.de}}\\


\begin{abstract}
\noindent In this paper, we consider the inverse shape problem of recovering isotropic scatterers with a conductive boundary condition. Here, we assume that the measured far-field data is known at a fixed wave number. Motivated by recent work, we study a new direct sampling indicator based on the Landweber iteration and the factorization method. Therefore, we prove the connection between these reconstruction methods. The method studied here falls under the category of  qualitative reconstruction methods where an imaging function is used to recover the absorbing scatterer. We prove stability of our new imaging function as well as derive a discrepancy principle for recovering the regularization parameter. The theoretical results are verified with numerical examples to show how the reconstruction performs by the new Landweber direct sampling method. 
\end{abstract}

\section{Introduction}
In this paper, we provide the analytical framework for recovering extended isotropic scatterers using a new direct sampling imaging function based on the Landweber regularization method. The isotropic scatterers have a conductive boundary condition that models an object that has a thin layer covering the exterior (such as an  aluminum sheet). For a fixed wave number we will assume that the so-called far-field pattern is measured. With this, we will use {\it qualitative reconstruction methods} to recover the unknown scatterer from the far-field pattern. Qualitative methods have been used in many inverse scattering problems \cite{ammari1,GLSM,fmconductbc,harris-nguyen,w/Jake,RTM-gibc,Liu,nf-meng} (and many more) due to the fact that they serve for nondestructive testing. The main advantage for using qualitative methods is the fact that little a priori information about the scatterer is needed. 
For many applications such as medical imaging this is very useful since one does not have much a priori knowledge of the unknown scatterer.

We consider reconstructing extended scatterers using an analogous method to the Direct Sampling Method (DSM). Here, we assume that we have the far-field operator i.e. we have the measured far-field pattern for all sources and receivers along the unit circle/sphere. See for e.g. \cite{BS1,Ito2012,Ito2013,BS4,Kang1,Kang2,BS5,liem,liem2} for the application of the direct sampling method for other inverse shape problems from scattering theory as well as \cite{BS2} for an application to diffuse optical tomography. In order to analyze the corresponding imaging function for the new Landweber direct sampling method, we will need a factorization of the far-field operator. Then by a similar analysis as is done in \cite{harris-dsm,Liu} we use the factorization and the Funk--Hecke integral identity to prove that the new imaging function will accurately recover the scatterer. 
This project was motivated by the works of \cite{harris-dsm} where a similar imaging function was introduced and analyzed for the behavior of the far-field operator associated with a non-absorbing scatterer. One of the main contributions is to connect the well known factorization method \cite{IST,FM-Book,MUSIC-kirsch} and the direct sampling method via the Landweber iteration.  In addition, we provide a direction on how to choose the regularization parameter as well as a stability result for the new imaging function. In this paper, we will analyze the imaging function corresponding to a polynomial approximate of the Landweber iteration solution operator associated with factorization method. This expands the ideas in \cite{harris-dsm} to also be valid when the scatterer has complex--valued coefficients. 

The rest of the paper is organized as follows. In the next section, we state the direct and inverse problem under consideration. We discuss the scattering by  an isotropic scatterers with a conductive boundary condition and set up the assumption for the scatterer. We then, study and derive a Lippmann-Schwinger integral equation for the scattered field. Then, we consider the factorization of the far-field operator and present some of the properties that the factors of it give us.  As a consequence, we then derive a Landweber iteration method that will establish the resolution analysis for the imaging function. Lastly, we present numerical examples based on the new sampling method based on the Landweber iteration. This shows that this method is analytically rigorous and computationally simple.

\section{Statement of the problem}\label{direct-prob-LDSM}
In this section, we formulate the direct scattering problem in $\mathbb{R}^d$ for extended isotropic scatterers with a conductive boundary condition where $d=2,3$. We assume that the scattering obstacle may be composed of multiple simply connected regions. For our model, we take an incident plane wave denoted $u^i$ to illuminate the scatterer. To this end, we let $u^i(x, \hat{y})=\text{e}^{\text{i}kx\cdot\hat{y}}$ where the incident direction $\hat{y}\in \mathbb{S}^{d-1}$(i.e. unit circle/sphere) and the point $x\in \mathbb{R}^d$. Notice, that the incident field $u^i$ satisfies 
$$\Delta u^i+k^2u^i=0 \quad \text{ in } \quad \mathbb{R}^d.$$ 
The interaction of the incident field and the scatterer denoted by $D$ produces the radiating scattered field $u^s(x,\hat{y})\in H^1_{loc}(\mathbb{R}^d)$ that satisfies 
\begin{align}
\Delta u^s+k^2 n(x)u^s=k^2 \big(1-n(x)\big) u^i \quad & \text{in} \hspace{.2cm}  \mathbb{R}^d \backslash \partial D \label{direct1}\\
u^s_{-} - u_{+}^s=0  \quad \text{ and } \quad \partial_\nu u^s_{-} -  \partial_\nu u^s_{+} = \eta(x) \big(u^s+u^i\big)  \quad & \text{on} \hspace{.2cm} \partial D. \label{direct2}
\end{align}
Along with the Sommerfeld radiation condition 
\begin{align}
{\partial_r u^s} - \text{i} ku^s =\mathcal{O} \left( \frac{1}{ r^{(d+1)/2} }\right) \quad \text{ as } \quad r \rightarrow \infty \label{direct3}
\end{align}
where $\partial_\nu \phi:= \nu \cdot \nabla \phi$ for any $\phi$ and  $r:=|x|$. The radiation condition \eqref{direct3} is satisfied uniformly in all directions $\hat{x}:=x/|x|$.  Here $-$ and $+$ corresponds to taking the trace from the interior or exterior of $D$, respectively. Note, that $u^i$ and its normal derivative are continuous across the boundary of $\partial D$. 

We assume here that the scatterer $D \subset \mathbb{R}^d$ has a boundary that is $C^2$ where $\nu$ denotes the unit outward normal vector to $\partial D$. For the material parameters, we assume that the refractive index satisfies $n \in L^{\infty}(D)$  such that supp$(n-1)=D$ where
$$Im(n(x))\geq0 \quad \text{ for a.e. } \quad x \in  D$$
and that the conductivity $\eta \in L^{\infty}(\partial D)$ satisfies that  
$$Im(\eta(x))\geq 0 \quad \text{ for a.e. } \quad x \in \partial D.$$
Here, the wave number $k>0$ is fixed and under the above assumptions we have that \eqref{direct1}--\eqref{direct3} is well-posed by \cite{fmconductbc}. 
Recall, the fundamental solution to the Helmholtz equation in $\R^d$ given by 
\begin{equation}\label{funsolu}
\Phi(x,y)=
    \begin{cases}
        \frac{\text{i}}{4} H_0^{(1)}(k|x-y|) & d=2\\
         \\ 
        \frac{\text{e}^{\text{i}k|x-y|}}{4\pi|x-y|} & d=3
    \end{cases}
\end{equation}
where $H^{(1)}_0$ denotes the first kind Hankel function of order zero. In our analysis, we will use the following asymptotic formula 
 $$\Phi(x,y)=\gamma \frac{\text{e}^{\text{i}k|x|}}{|x|^{(d-1)/2}} \left\{ \text{e}^{-\text{i}\hat{x}\cdot y} +\mathcal{O}\left( \frac{1}{|x|} \right) \right\} $$ 
 as $|x| \longrightarrow \infty$ uniformly with respect to $\hat{x}$. Here, the parameter  
$$ \gamma=\frac{\text{e}^{\text{i}\pi/4}}{\sqrt{8\pi k}}\hspace{.3cm}\text{in}\hspace{.3cm}\mathbb{R}^2\hspace{.3cm}\text{and}\hspace{.3cm}\gamma=\frac{1}{4\pi}\hspace{.3cm}\text{in}\hspace{.3cm}\mathbb{R}^3.$$
Due to the fact that $u^s$ is a radiating solution to the Helmholtz equation on the exterior of $D$, similarly we have that (see for e.g. \cite{approach,IST})
$$u^s(x,\hat{y})=\gamma \frac{\text{e}^{\text{i}k|x|}}{|x|^{(d-1)/2}}\left\{u^{\infty}(\hat{x},\hat{y})+\mathcal{O}\left(\frac{1}{|x|}\right) \right\}\hspace{.3cm}\text{as}\hspace{.3cm}|x|\longrightarrow \infty.$$
Again, we have that the asymptotic formula holds uniformly with respect to $\hat{x}$. The function $u^{\infty}(\hat{x},\hat{y})$ denotes the far-field pattern of the scattered field for observation direction $\hat{x}$ and incident direction $\hat{y}$ on $\mathbb{S}^{d-1}$. We can now define the far-field operator denoted $F$ given by 
\begin{equation}\label{ff-operator}
(Fg)(\hat{x})=\int_{\mathbb{S}^{d-1}}u^{\infty}(\hat{x},\hat{y})g(\hat{y})\, \text{d}s(\hat{y}) \quad \text{ for } \quad g\in L^2(\mathbb{S}^{d-1})
\end{equation}
mapping $L^2(\mathbb{S}^{d-1})$ into itself.

We are interested in using the known and measured far-field pattern to recover the unknown absorbing scatterer $D$.  Now, we note that the fundamental solution satisfies 
$$\Delta\Phi(x,\cdot)+k^2\Phi(x,\cdot)=-\delta(x-\cdot)\hspace{.3cm}\text{in}\hspace{.3cm}\mathbb{R}^d$$
along with the Sommerfeld radiation condition \eqref{direct3}. Using Green's 2nd Theorem when $x$ is in the interior of $D$ gives that 
\begin{align*}
   u^s(x)\chi_D&=-\int_D \Phi(x,z) [\Delta u^s(z)+k^2 u^s(z)] \, \text{d}z\\
&\hspace{0.5in}+\int_{\partial D} \Phi(x,z)\frac{\partial }{\partial \nu}u^s_-(z)-u^s_-(z)\frac{\partial }{\partial \nu}\Phi(x,z)\, \text{d}s(z)
\end{align*}
where $\chi_D$ is the indicator function on the scatterer $D$. In a similar manner, using Green's 2nd Theorem when $x$ is in the exterior of $D$  gives that
\begin{align*}
    u^s(x)(1-\chi_D)&=-\int_{\partial D} \Phi(x,z)\frac{\partial}{\partial \nu}u^s_+(z)-u^s_+(z)\frac{\partial}{\partial\nu}\Phi(x,z) \, \text{d}s(z)\\
    &\hspace{0.5in}+\int_{\partial B_R} \Phi(x,z)\frac{\partial}{\partial \nu}u^s(z)-u^s(z)\frac{\partial}{\partial \nu}\Phi(x,z)\, \text{d}s(z).
\end{align*}
where $B_R=\{ x \in \R^d \, \, : \, \, |x|<R \}$ such that $D \subset B_R$. Note, that we have used the fact that $u^s$ is a solution to the Helmholtz equation on the exterior of $D$.

Therefore, by adding the above expressions we obtain the Lippmann-Schwinger type representation of the scattered field 
\begin{align}
u^s(x)&= k^2 \int_{D} (n-1)\Phi(x,z)\big[u^s(z)+u^i(z)\big] \, \text{d}z + \int_{\partial D} \eta\Phi(x,z)\big[u^s(z)+u^i(z)\big] \, \text{d}s(z). \label{totalfield}
\end{align}
Notice, that we have used the scattered field and fundamental solution satisfying the radiation condition \eqref{direct3} to handle the boundary integral over $\partial B_R$ by letting $R \to \infty$. In the following section we will define a far-field pattern that involves the parameter at the boundary and factorize it. We will analyze equation \eqref{totalfield} in order to better understand the behavior of the scattered field.

\section{Factorizing the Scattered Field}\label{factor-LDSM}
In this section, we study an extension of the direct sampling method to solve the problem for the reconstruction of absorbing scatterers. The Lippmann--Schwinger representation of the scattered field \eqref{totalfield} will be used in our analysis. We will derive a new factorization of the far-field operator defined in \eqref{ff-operator} which is one of the main components of our analysis. We will prove that the new proposed imaging function has the property that it decays as the sampling point moves away from the scatterer. 

We begin by factorizing the far-field operator defined in \eqref{ff-operator} which will allow us to define an imaging function to facilitate the reconstruction of extended regions $D$. Recall, that the far-field operator for $g\in L^2(\mathbb{S}^{d-1})$ is given by 
$$(Fg)(\hat{x})=\int_{\mathbb{S}^{d-1}}u^\infty(\hat{x},\hat{y})g(\hat{y})ds(\hat{y})$$
where $\mathbb{S}^{d-1}$ is the unit sphere/circle. Since it is well-known that the far-field pattern is analytic (see for e.g. \cite{coltonkress}) it is clear that $F$ is a compact operator. 
It has been shown in \cite{fmconductbc}, that the far-field operator is injective with a dense range provided that
\begin{align}
\Delta \varphi +k^2 \varphi = 0 \quad \text{ and } \quad  \Delta \phi +k^2\phi = 0  \hspace{.2cm}& \text{in} \hspace{.2cm}  D \label{eigproblem1}\\
\varphi=\phi \quad \text{ and } \quad   \partial_{\nu}\varphi=\partial_{\nu}\phi+\eta \phi \label{eigproblem2} \hspace{.2cm}& \text{on} \hspace{.2cm}  \partial D
\end{align}
only admits the trivial solution in $L^2(D) \times L^2(D)$. This says that the wave number $k$ is not a transmission eigenvalue. This problem has been studied \cite{te-cbc,te-cbc2,te-cbc3} and it is known that the set of transmission eigenvalues is at most discrete in the complex plane, provided that $|n-1|^{-1} \in L^\infty (D)$ and $\eta^{-1} \in L^\infty (\partial D)$ (see also \cite{te-2cbc} for a recent study with two conductivity parameters). Therefore, we will make the assumption that \eqref{eigproblem1}--\eqref{eigproblem2} only admits the trivial solution. 
The factorization of the far-field operator was initially studied \cite{Liu} for the case when $\eta=0$. Now, we recall the Lippmann--Schwinger representation of the scattered field 
$$u^s(x,\hat{y})=\int_{D}k^2(n-1)(u^s+u^i)\Phi(x,z)dz+\int_{\partial D}\eta(u^i+u^s)\Phi(x,z)ds(z)$$
which implies that 
$$u^\infty(\hat{x},\hat{y})=\int_{D}k^2(n-1)(u^s+u^i) \text{e}^{-\text{i}k\hat{x}\cdot z}\, \text{d}z+\int_{\partial D}\eta(u^i+u^s)\text{e}^{-\text{i}k\hat{x}\cdot z}\, \text{d}s(z).$$
Using the above formula for the far-field pattern, we can change the order of integration to obtain the following identity 
$$Fg=\int_{D}k^2(n-1)(u_g^s+v_g) \text{e}^{-\text{i}k\hat{x}\cdot z} \, \text{d}z+\int_{\partial D}\eta(u_g^s+v_g) \text{e}^{-\text{i}k\hat{x}\cdot z}\, \text{d}s(z).$$
Here, we let $v_g(x)$ denote the Herglotz wave function defined as  
$$v_g (x) =\int_{\mathbb{S}^{d-1}} \text{e}^{\text{i}k{x}\cdot \hat{y}}g(\hat{y}) \,  \text{d}s(\hat{y}) \quad \text{ and } \quad u_g^{s} (x)=\int_{\mathbb{S}^{d-1}} u^s(x,\hat{y})g(\hat{y}) \, \text{d}s(\hat{y})$$
where $u_g^s$ solves the boundary value problem \eqref{direct1}--\eqref{direct3} when the incident field $u^i=v_g$.

The factorization method for the far-field operator $F$ is based on factorizing $F$ into three distinct pieces that act together and give us more information about the region of interest $D$. To this end, one can show that  
\begin{equation} \label{Hdef}
H:L^2(\mathbb{S}^{d-1})\longrightarrow L^2(D)\times L^2(\partial D) \quad \text{where} \quad Hg=(v_g|_{D} \, ,\, v_g|_{\partial D})
\end{equation}
is a bounded linear operator. Now, we consider the following auxiliary problem  
\begin{align}
(\Delta +k^2 n)w=-k^2(n-1)f \hspace{.2cm}& \text{in} \hspace{.2cm}  \mathbb{R}^d\setminus \partial D\label{aux1}\\
w_{-} - w_{+}=0 \quad  \text{and } \quad \partial_{\nu}w_{-} - \partial_\nu w_{+}= \eta(w+h) \quad &\text{on } \quad \partial D \label{aux2}\\
 {\partial_r w} - \text{i} k w =\mathcal{O} \left( \frac{1}{ r^{(d+1)/2} }\right) \quad& \text{ as } \quad r \rightarrow \infty \label{aux3}
\end{align}
with $f\in L^2(D)$ and $h\in L^2(\partial D)$. It is clear that the auxiliary problem \eqref{aux1}--\eqref{aux3} is well-posed by \cite{fmconductbc} with $w \in H^1_{loc}(\R^d)$ under the assumptions of this paper. We can define the operator $T$ associated with the auxiliary problem \eqref{aux1}--\eqref{aux3} such that  
$$T:L^2(D)\times L^2(\partial D)\longrightarrow L^2(D)\times L^2(\partial D)$$ 
which is given by   
\begin{equation}
    T(f,h)=\left(k^2(n-1)(w+f)|_{D} \, , \, \eta(w+h)|_{\partial D}\right). \label{defT}
\end{equation}
Similarly, we have that  $T$ is a bounded and linear operator. Due to the fact that $u_g^s$ solves \eqref{aux1}--\eqref{aux3} with $f=v_g |_{D}$ and $h = v_g |_{\partial D}$, we have that 
$$THg = \left(k^2(n-1)(u^s_g+v_g)|_{D} \, , \, \eta(u^s_g+v_g)|_{\partial D}\right)$$
for any $g \in L^2(\mathbb{S}^{d-1})$.

Now, in order to determine a suitable factorization of the far-field operator $F$, we need to compute the adjoint of the operator $H$. Observe that by definition we have 
\begin{align*}
    \left(Hg,(\varphi_1,\varphi_2)\right)_{L^2(D)\times L^2(\partial D)} & \\
    &\hspace{-1.25in}=\int_D\left(\int_{\mathbb{S}^{d-1}} \text{e}^{\text{i}k\hat{x}\cdot z}g(\hat{x}) \, \text{d}s(\hat{x})\right)\overline{\varphi_1}(z) \, \text{d}z+\int_{\partial D}\left(\int_{\mathbb{S}^{d-1}}\text{e}^{\text{i}k\hat{x}\cdot z}g(\hat{x})\, \text{d}s(\hat{x})\right)\overline{\varphi_2}(z) \, \text{d}s(z)\\
    &\hspace{-1.25in}=\int_{\mathbb{S}^{d-1}}g(\hat{x})\left(\int_D  \text{e}^{\text{i}k\hat{x}\cdot z}\overline{\varphi_1}(z) \, \text{d}z\right)\, \text{d}s(\hat{x})+\int_{\mathbb{S}^{d-1}}g(\hat{x})\left(\int_{\partial D} \text{e}^{\text{i}k\hat{x}\cdot z}\overline{\varphi_2}(z) \, \text{d}z\right)\, \text{d}s(\hat{x})\\
    &\hspace{-1.25in}=\int_{\mathbb{S}^{d-1}}g(\hat{x})\overline{\left(\int_D  \text{e}^{-\text{i}k\hat{x}\cdot z}{\varphi_1}(z) \, \text{d}z\right)}\, \text{d}s(\hat{x})+\int_{\mathbb{S}^{d-1}}g(\hat{x})\overline{\left(\int_{\partial D} \text{e}^{-\text{i}k\hat{x}\cdot z}{\varphi_2}(z) \, \text{d}z\right)}\, \text{d}s(\hat{x})
\end{align*}
and thus 
$$H^*(\varphi_1,\varphi_2)=\int_D\varphi_1 (z) \text{e}^{-\text{i}k\hat{x}\cdot z}\, \text{d}z+\int_{\partial D}\varphi_2 (z)  \text{e}^{-\text{i}k\hat{x}\cdot z}\, \text{d}s(z).$$
We then obtain
$$H^*THg=\int_{D}k^2(n-1)(u_g^s+v_g) \text{e}^{-\text{i}k\hat{x}\cdot z} \, \text{d}z+\int_{\partial D}\eta(u_g^s+v_g) \text{e}^{-\text{i}k\hat{x}\cdot z}\, \text{d}s(z)=Fg$$
for any $g \in L^2(\mathbb{S}^{d-1})$. Therefore, we have derived a factorization for the far-field operator. 

\begin{theorem}
The far-field operator $F:L^2(\mathbb{S}^{d-1})\longrightarrow L^2(\mathbb{S}^{d-1}) $ has the symmetric factorization $F=H^*TH$ where the operators $H$ and $T$ are defined in \eqref{Hdef} and \eqref{defT}, respectively. 
\label{Ffact}
\end{theorem}
The factorization given above is one of the main pieces that will be used to derive an imaging functional. The next step in our analysis is the Funk--Hecke integral identity, this integral identity gives us the opportunity to evaluate the Herglotz wave function for $g = \phi_z$ which is given by 
\begin{equation}
v_{\phi_z} (x)= \int_{\mathbb{S}^{d-1}} \text{e}^{-\text{i}k(z-x)\cdot \hat{y}}ds(\hat{y})=
    \begin{cases}
        2\pi J_0(k|x-z|), & \text{in } \mathbb{R}^2,\\
        \\
        4\pi j_0(k|x-z|), & \text{in } \mathbb{R}^3.
    \end{cases} \label{Herglotz}
\end{equation}
Here, $J_0$ is the zeroth order Bessel function of the first kind and $j_0$ is the zeroth order spherical Bessel function of the first kind. With the factorization of the far-field operator $F$ and the Funk--Hecke integral identity, we can solve the inverse problem of recovering $D$ by using the decay of the Bessel functions (similarly done in \cite{nf-harris,harris-liem}).   
 
The final piece needed in our study for the factorization of $F$ is to analyze the middle operator $T$. Just as in the factorization \cite{fmconductbc,FM-Book,MUSIC-kirsch} and generalized linear sampling methods \cite{GLSM,phong1,GLSM2}, coercivity of the middle operator is essential to our analysis and will be proven to gather information about the far-field operator $F$. To this end, we will show that $T$ is coercive with respect to the $\overline{Range(H)}$. 
Thus, we begin by showing that $T$ can be decomposed into a sum of a compact and coercive operator.

\begin{theorem}\label{Tdecom}
Let the operator $T$ be as defined in \eqref{defT}. Then we have that $T=S+K$ where the operators $S$ and $K:L^2(D)\times L^2(\partial D)\longrightarrow L^2(D)\times L^2(\partial D)$ are given by 
$$S(f,h)=(k^2(n-1)f,\eta h) \quad \text{ and } \quad K(f,h)=(k^2(n-1)w,\eta w)$$
where $w$ is the unique solution to \eqref{aux1}--\eqref{aux3}.  Moreover, we have that $K$ is a compact operator and 
$$ \pm Re\left(S(f,h),(f,h)\right)_{L^2(D)\times L^2(\partial D)}  \geq \alpha \| (f,h) \|^2_{L^2(D)\times L^2(\partial D)}  \quad \text{for some positive $\alpha>0$}$$
provided that $\pm Re(n-1)$ and $\pm Re(\eta)$ are positive definite.
\end{theorem}
\begin{proof}
To prove the claim, we first start with proving the coercivity result for the operator $S$. Therefore, by definition of the operator we have that 
$$\left( S(f,h),(f,h) \right)_{L^2(D)\times L^2(\partial D)} = k^2 \left((n-1)f,f \right)_{L^2(D)} + (\eta h, h )_{L^2(\partial D)}$$
and by the assumptions on the coefficients we can obtain the estimate
$$\pm Re\left( S(f,h),(f,h) \right)_{L^2(D)\times L^2(\partial D)}  \geq \alpha ||(f,h)||^2_{L^2(D)\times L^2(\partial D)}$$
for some constant $\alpha>0$ depending on the coefficients. 
 
 Now, the compactness of the operator $K$ is due to the fact that $H^{1/2}(\partial D)$ is compactly embedded into $L^2(\partial D)$ as well as $H^1(D)$ being compactly embedded into $L^2(D)$.  This proves the claim. 
\end{proof}

Now, we proceed with stating a well-known limit (see for e.g. \cite{FM-Book}) that will help us analyze the behavior of the imaginary part of the operator $T$. Studying the imaginary part  of $T$ will help prove our coercivity result for the operator. Let $w$ be the solution function of the auxiliary problem above, i.e. \eqref{aux1}--\eqref{aux3}. Then, we have that
\begin{equation}
\int_{\partial B_R}\overline{w}\partial_{\nu}w \, \text{d}s\longrightarrow \text{i} |\gamma|^2 k \int_{\mathbb{S}^{d-1}}|w^{\infty}|^2 \, \text{d}s \quad \text{ as } \quad R\longrightarrow \infty. \label{partialw}
\end{equation}
 We can now prove that the imaginary part of the operator $T$ is positive on the $\overline{Range(H)}$. Recall, that we have assumed that the material parameters satisfy the estimates $Im(n)\geq 0$ and $Im(\eta) \geq 0$. 

\begin{theorem} \label{imT}
Let the operator $T$ be as defined in \eqref{defT}. Then  we have that 
$$Im\left(T(f,h),(f,h)\right)_{{L^2(D)}\times{L^2(\partial D)}} > 0$$ 
for all $(0,0) \neq (f,h) \in \overline{Range(H)}$ provided that $k$ is not a transmission eigenvalue. 
\end{theorem}
\begin{proof}
In order to prove the claim, we will express the inner-product using the auxiliary boundary value problem \eqref{aux1}--\eqref{aux3} for $w$ with inputs $f$ and $h$. We begin, by using the fact that 
$$\overline{f}=\overline{f+w}-\overline{w} \quad \text{ as well as } \quad \overline{h}=\overline{h+w}-\overline{h}$$ 
and observe that 
\begin{align*}
    \left(T(f,h),(f,h)\right)_{{L^2(D)}\times{L^2(\partial D)}}&=\int_D k^2(n-1)|w+f|^2dx+\int_{\partial D}\eta |h+w|^2 \, \text{d}s\\
    &-\int_Dk^2(n-1)(w+f)\overline{w}\, \text{d}x-\int_{\partial D}\eta (w+h)\overline{w}\, \text{d}s.
\end{align*}
Recall, our auxiliary problem \eqref{aux1}--\eqref{aux3} and since $n=1$ on the exterior of $D$, we have 
\begin{equation}
\Delta w+k^2w=-k^2(n-1)(w+f) \hspace{.5cm}\text{in}\hspace{.5cm} B_R\setminus \partial D \hspace{.5cm}\text{for}\hspace{.5cm} D\subset B_R. \label{n=1}
\end{equation}
Then, we apply Green's 1st Theorem on \eqref{n=1} in $D$ to obtain  
\begin{align*} 
- \int_D k^2(n-1)(w+f)\overline{w} \, \text{d}x &= \int_D  \overline{w}(\Delta w+k^2w)\, \text{d}x\\
						    &=-\int_D |\grad w|^2 -k^2 |w|^2 \, \text{d}x + \int_{\partial D} \overline{w} \partial w_{-} \, \text{d}s 
\end{align*} 
and in $B_R \setminus \partial D$ we have that 
\begin{align*} 
0 = -\int_{B_R \setminus \overline{D}}  |\grad w|^2 -k^2 |w|^2 \, \text{d}x + \int_{\partial B_R} \overline{w} \partial_\nu w \, \text{d}s - \int_{\partial D} \overline{w} \partial w_{+} \, \text{d}s.
\end{align*} 
Now, by  the appealing to the jump in the normal derivative
$$\partial_{\nu}w_{-} - \partial_\nu w_{+}= \eta(w+h) \quad \text{on } \quad \partial D$$
we have the equality 
\begin{align*}
    \left(T(f,h),(f,h)\right)_{{L^2(D)}\times{L^2(\partial D)}}&=\int_Dk^2(n-1)|w+f|^2 \, \text{d}x+\int_{\partial D}\eta |h+w|^2 \, \text{d}s\\
    & -\int_{B_R}|\nabla w|^2-k^2|w|^2 \, \text{d}x+\int_{\partial B_R}\overline{w}\partial_{\nu}w \, \text{d}s.
\end{align*}
Letting $R\longrightarrow \infty$ and using \eqref{partialw} we see that 
\begin{align}
   Im\left(T(f,h),(f,h)\right)_{{L^2(D)}\times{L^2(\partial D)}}&=k^2\int_D Im(n)|w+f|^2\, \text{d}x \nonumber \\
    &+\int_{\partial D}Im(\eta)|h+w|^2 \, \text{d}s+|\gamma|^2 k \int_{\mathbb{S}^{d-1}}|w^{\infty}|^2 \, \text{d}s. \label{imTequ}
\end{align}
By assumption on the imaginary part of the coefficients, we have that the imaginary part of $T$ is non-negative in $L^2(D) \times L^2(\partial D)$.

Now, we prove that imaginary part of $T$ is positive in $\overline{Range(H)}$. To this end, we assume that there exists $(f,h) \in \overline{Range(H)}$ such that 
$$Im\left(T(f,h),(f,h)\right)_{{L^2(D)}\times{L^2(\partial D)}}=0$$
and we must prove that prove $(f,h) = (0,0)$. 
From the definition of $H$, we have that $f=v|_{D}$ and $h=v|_{\partial D}$ where $v$ is a solution to the Helmholtz equation in $D$. 
Notice, that by \eqref{imTequ} we have that $w^{\infty}=0$ and by Rellich's Lemma (see for e.g. \cite{approach,IST}) we have that $w = 0$ in $\R^d \setminus \overline{D}$. By the boundary conditions in \eqref{aux2}, we have that 
$$w_{-} = 0\quad \text{ and } \quad \partial_{\nu}w_{-} = \eta v \quad \text{on } \quad \partial D$$
since $\partial_{\nu}w_{+} =w_{+} = 0$ on $\partial D$. We also have that 
$$\Delta w+k^2w=-k^2(n-1)(w+v) \quad \text{ and }\quad \Delta v +k^2v=0 \quad \text{in } \quad D.$$
Combining the above inequalities, $(w+v , v) \in L^2(D) \times L^2(D)$ satisfy the boundary value problem 
\begin{align*}
\Delta (w+v)+k^2 n (w+v)=0 \quad \text{ and } \quad \Delta v+k^2v=0 \quad &\text{in } \quad D\\
(w+v) = v \quad \text{ and } \quad \partial_{\nu}(w+v) = \partial_\nu v + \eta v \quad &\text{on } \quad \partial D.
\end{align*}
By our assumption, we have that the above boundary value problem only admits the trivial solution $(w+v , v)=(0,0)$ i.e. $f=0$ and $h=0$, proving the claim. 
\end{proof}

In the next section we will use this factorization to derive a direct sampling method that is connected to the factorization method.

 \section{The Landweber Direct Sampling Method}
 In this section, we study a Landweber indicator function using the operator $F_{\#}$ defined below. In previous works, a similar reconstruction method for extended regions based on \eqref{totalfield} was studied for the case where $\eta=0$ and $n$ real-valued see \cite{harris-dsm}. Although, the authors did not consider the case of absorbing scatterers they got better reconstructions using the Tikhonov direct sampling method. In our problem, the coefficients for the scattering problem \eqref{direct1}--\eqref{direct3} are complex and thus we must use a different characterization for the operator.  We extend the regularization to the Landweber iteration basing it on the factorization method. In comparison to previous studies, the Landweber iteration will provide us the ability to pick a regularization parameter considering a discrepancy principle and the `optimal' number of iterations. 
 
The operator is defined to be $F_{\#}=|\text{Re}(F)|+| \text{Im}(F)|$ where 
$$\text{Re}({F})=\frac{1}{2}({F}+{F}^{*}) \quad \text{and} \quad \text{Im}({F})=\frac{1}{2 \text{i}}({F}-{F}^{*}).$$ 
Note, that the absolute value of the above self-adjoint compact operators is given by its eigenvalue decomposition. 
One can easily show that $F_{\#}$ is a self-adjoint, compact, and positive (see for e.g. \cite{FM-Book}). Therefore, we have that the operator $F_{\#}$ have an orthonormal eigenvalue decomposition $(\lambda_j,\psi_j)\in \mathbb{R}_+\times L^2(\mathbb{S}^{d-1})$ such that 
 $$F_{\#}g=\sum_{j=1}^{\infty}\lambda_j(g,\psi_j)_{L^2(\mathbb{S}^{d-1})}\psi_j\hspace{.5cm}\text{for all}\hspace{.5cm}g\in L^2(\mathbb{S}^{d-1}).$$
 As a consequence of $F_{\#}$ being a compact operator we have that $\lambda_j\longrightarrow 0$ as $j\longrightarrow \infty.$ Thus, we have that $\lambda_j\neq 0$ for all $j$ and the set $\{\psi_j\}$ is a complete orthonormal set in $L^2(\mathbb{S}^{d-1}).$

\subsection{Derivation of the Landweber Regularization}
\label{Landweberderi}
We use the operator $F_{\#}$ to recover absorbing scatteres by solving the ill-posed equation of the form 
\begin{equation}\label{FFE-1CBC}
    F_{\#}^{1/2}g_z=\phi_{z}\hspace{.25cm}\text{for}\hspace{.25cm} z\in \mathbb{R}^d
\end{equation}
which is solvable if and only if the sampling point $z\in D.$ We will derive an approximate solution operator to the above equation and use the Landweber iteration to approximate the solution operator. We  exploit the fact that we can construct a polynomial that when applied to the operator acts as the solution operator for $F_{\#}$. 

The Landweder regularized solution to \eqref{FFE-1CBC} will be denoted $g_z^r$ and using the eigenvalue decomposition we have that 
$$g_z^r=\sum_{j=1}^{\infty}\frac{1}{\sqrt{\lambda_j}}\Big[1-(1-\beta\lambda_j)^r\Big](\phi_z,\psi_j)_{L^2(\mathbb{S}^{d-1})}\psi_j.$$
We define the filter function 
$$\Gamma_r(t)=\frac{1-(1-\beta t)^r}{\sqrt{t}} \quad \text{ where} \quad  \beta\in (0,{1}/{\lambda_1})  \quad \text{ and} \quad  r\in \mathbb{N}$$
 which has a removable discontinuity such that $\Gamma_r(0):=0\, $ and as a consequence is continuous on the interval $[0, \lambda_1]$. The function $\Gamma_r(t)$ is connected to the solution operator for the Landweder regularization given by the mapping
\begin{equation}
\phi_z \longmapsto \sum_{j=1}^{\infty} \Gamma_r(\lambda_j) (\phi_z,\psi_j)_{L^2(\mathbb{S}^{d-1})}\psi_j. \label{mappphiz}
\end{equation}
We note that the parameter $\beta$ is chosen to be in the interval $\beta\in (0,{1}/{\lambda_1})$ and $r\in \mathbb{N}$ that we can control and choose throughout the calculations in our experiments. 

In order to approximate our solution operator \eqref{mappphiz}, we will exploit the fact that the function $\Gamma_r(t)$ is continuous for all $t\geq 0$. To this end, for every $\epsilon>0$ there is a polynomial $P_{r,\epsilon}(t)$ where $P_{r,\epsilon}(0)$=0 that approximates our function $\Gamma_r(t)$ such that 
\begin{equation}
    \left\|P_{r,\epsilon}(t)-\Gamma_r(t)\right\|_{L^{\infty}(0,\lambda_1)}<\epsilon.
    \label{approxpoly}
\end{equation}
The construction of this approximation polynomial gives us an approximation of the solution operator that is defined to be 

\begin{equation}
P_{r,\epsilon}(F_{\#})\phi_z=\sum_{j=1}^{\infty}P_{r,\epsilon}(\lambda_j)(\phi_z,\psi_j)\psi_j.    \label{polyF}
\end{equation}
Using \eqref{mappphiz} and \eqref{approxpoly}, we propose a Landweder indicator function for a fixed $r$ and $\beta$ and this is by exploiting the defined polynomial of the operator $F_{\#}$ via the eigenvalue decomposition as commonly done in linear algebra. Thus we have the following imaging function 
\begin{equation}
W_{\text{LDSM}}(z)=\left\|P_{r,\epsilon}(F_\#)\phi_z\right\|^2_{L^2(\mathbb{S}^{d-1})} \quad  \text{with} \quad \left\|P_{r,\epsilon}(t)-\Gamma_r(t)\right\|_{L^{\infty}(0,\lambda_1)}\approx 0    \label{DSMdef}
\end{equation}
where $P_{r,\epsilon}(t)$ is our approximation polynomial. 

We know that $\{\psi_j\}$ is an orthonormal basis in the space $L^2({\mathbb{S}^{d-1}})$ and as a consequence using the definition of $P_{r,\epsilon}(F_\#)\phi_z$ we have 
$$\left\|P_{r,\epsilon}(F_\#)\phi_z\right\|^2_{L^2({\mathbb{S}^{d-1}})}=\sum_{j=1}^{\infty}P^2_{r,\epsilon}(\lambda_j)\Big|(\phi_z,\psi_j)_{L^2({\mathbb{S}^{d-1}})}\Big|^2.$$
Now that we have the new Landweder indicator function $W_{\text{LDSM}}(z)$ where we will connect it to the factorization operator derived in Section \ref{factor-LDSM}. Observe that \eqref{approxpoly} gives us the following inequality for all $\epsilon>0$ and for fixed parameters $\beta$ and $r$ 
$$P^2_{r,\epsilon}(\lambda_j)\leq \Gamma^2_{r}(\lambda_j)+2\epsilon\Gamma_r(\lambda_j)+\epsilon^2,$$
where this holds for all $\epsilon>0$ and $j\in \mathbb{N}.$ We know that $\Gamma_r(t)$ is continuous for all $t\geq 0$ and by using Bernoulli's inequality for $t \geq -1$ we have $\Gamma_r(t)\leq r\beta \sqrt{t}.$ As a consequence, we define 
$$C_r:=\text{sup}_{(0, \lambda_1)}\left\|\Gamma_r(t)\right\|\leq \text{sup}_{(0, \lambda_1)}r\beta\sqrt{t}=r\beta\sqrt{ \lambda_1 }<\infty$$
where it only depends on $\beta,r,$ and $\lambda_1.$ 
Now, take $0<\epsilon <1$ and we estimate the following
\begin{align*}
    \left\|P_{r,\epsilon}(F_\#)\phi_z\right\|^2_{L^2({\mathbb{S}^{d-1}})}&\leq \sum_{j=1}^{\infty}\Gamma_r^2(\lambda_j)\Big|(\phi_z,\psi_j)_{L^2({\mathbb{S}^{d-1}})}\Big|^2\\
    &+2\epsilon\sum_{j=1}^{\infty}\Gamma_r(\lambda_j)\Big|(\phi_z,\psi_j)_{L^2({\mathbb{S}^{d-1}})}\Big|^2+\epsilon^2\sum_{j=1}^{\infty}\Big|(\phi_z,\psi_j)_{L^2({\mathbb{S}^{d-1}})}\Big|^2\\
    &\leq \sum_{j=1}^{\infty}\Gamma_r^2(\lambda_j)\Big|(\phi_z,\psi_j)_{L^2({\mathbb{S}^{d-1}})}\Big|^2+(2\epsilon C_r+\epsilon^2)||\phi_z||^2_{L^2({\mathbb{S}^{d-1}})}\\
    &=\sum_{j=1}^{\infty}\Gamma_r^2(\lambda_j)\Big|(\phi_z,\psi_j)_{L^2({\mathbb{S}^{d-1}})}\Big|^2+2^{d-1}\pi(2\epsilon C_r+\epsilon^2)
\end{align*}
where $\left\|\phi_z\right\|^2_{L^2({\mathbb{S}^{d-1}})}=2^{d-1}\pi.$ Using Bernoulli's inequality once more we can easily see  $\Gamma_r^2(\lambda_j)\leq r^2\beta^2\lambda_j$ and combining this bound with the above inequalities gives

$$ \left\|P_{r,\epsilon}(F_\#)\phi_z\right\|^2_{L^2({\mathbb{S}^{d-1}})}\leq r^2\beta^2\sum_{j=1}^{\infty}\lambda_j\Big|(\phi_z,\psi_j)_{L^2({\mathbb{S}^{d-1}})}\Big|^2+C(r,d)\epsilon$$
where $C(r,d)$ is a positive constant depending on our regularization parameters and the dimension. The definition of $F_{\#}^{1/2}$ and our above inequalities implies that 
$$ \left\|P_{r,\epsilon}(F_\#)\phi_z\right\|^2_{L^2({\mathbb{S}^{d-1}})}\leq r^2\beta^2\Big|(F_{\#}\phi_z,\phi_z)_{L^2({\mathbb{S}^{d-1}})}\Big|+C(r,d)\epsilon$$
for fixed $\beta$ and $r.$ 

 In the previous section, Theorem \eqref{Ffact} established a factorization of the operator $F$. Now, with our new operator $F_{\#}$, it is known that by Theorem \ref{Tdecom} and \ref{imT} that the operator $F_{\#}=H^*T_{\#}H$ where the new operator $T_{\#}$ is coercive. Having this factorization allows us to do the following

\begin{align*}
    \Big|(F_{\#}\phi_z,\phi_z)_{L^2(\mathbb{S}^{d-1})}\Big|&=\Big|(T_{\#}H\phi_z,H\phi_z)_{L^2(D)\times L^2(\partial D)}\Big|.
\end{align*}
Thus, there exists constants $c_1$ and $c_2$ such that 
\begin{align}\label{imgfuncbounds}
c_1 \left( \| v_{\phi_z}\|^2_{L^2(D)} + \| v_{\phi_z}\|^2_{L^2(\partial D)} \right) &\leq  \left|(F_{\#}\phi_z,\phi_z)_{L^2(\mathbb{S}^{d-1})}\right| \nonumber \\
&\hspace{1in}\leq c_2  \left( \| v_{\phi_z}\|^2_{L^2(D)} + \| v_{\phi_z}\|^2_{L^2(\partial D)} \right).
\end{align}
Thus we have the main result of this section which relates the operator $F_{\#}$ to the Bessel functions that will decay as we move far away from the region of interest.

\begin{theorem}\label{distD}
For all $z\in \mathbb{R}^d\setminus \overline{D}$ we have that 
$$\left\|P_{r,\epsilon}(F_\#)\phi_z\right\|^2_{L^2({\mathbb{S}^{d-1}})}\leq C\text{dist}(z,D)^{1-d}+\mathcal{O}(\epsilon)\hspace{.5cm}\text{for}\hspace{.5cm} \text{dist}(z,D)\to \infty$$
where the $P_{r,\epsilon}(t) =\Gamma_r(t) + \mathcal{O}(\epsilon)$ as $\epsilon \longrightarrow 0$.
\end{theorem}
\begin{proof}
The proof of the claim is a result of the fact that 
$$ \left\|P_{r,\epsilon}(F_\#)\phi_z\right\|^2_{L^2({\mathbb{S}^{d-1}})}\leq r^2\beta^2\Big|(F_{\#}\phi_z,\phi_z)_{L^2({\mathbb{S}^{d-1}})}\Big|+\mathcal{O}(\epsilon)$$
as $\epsilon\to 0$ along with equations \eqref{Herglotz} and \eqref{imgfuncbounds}. Then, by using the fact that the Bessel function $J_0 (|z-x|)$ decays at a rate of $|z-x|^{-1/2}$ as $|z-x| \to \infty$ for $d=2$ and $j_0 (|z-x|)$ decays at a rate of $|z-x|^{-1}$ as $|z-x| \to \infty$ for $d=3$.
\end{proof}

This theorem gives the resolution analysis for using the imaging function. This implies that the imaging function will decay fast when we move away from the scatterer. Also, an important question about developing this imaging function is the choice and control over the parameter $r\in \mathbb{N}$. We present a discrepancy principle to determine $r$ and also an stability result for the new imaging function $W_{\text{LDSM}}(z)$ given by \eqref{DSMdef}.

\subsection{Determination of the parameter $r\in \mathbb{N}$ and stability result}
\label{rdefine}
Here we will assume that we have the perturbed far-field operator $F^\delta = F+\mathcal{O}(\delta)$ as $\delta \to 0$. The known $\delta \in (0,1)$ represents the noise level from our measured far-field data. Now, that we have derived our new sampling method we consider the imaging function where we use $F^\delta_{\#}$, as well as address how to determine the parameter $r\in \N$. To this end, we develop a discrepancy principle using the principle eigenvalue $\lambda_1$. We consider solving
\begin{equation}
  \Gamma_r(\lambda_1)-\Gamma_{r+1}(\lambda_1)=\delta 
  \label{gammateps}
\end{equation}
for $r$ i.e. we use $r$ iterations until we hit the noise level. Solving for $r$ in \eqref{gammateps} gives us that 
$$r=\frac{\ln\Big(\frac{\delta}{\beta\sqrt{\lambda_1}}\Big)}{\ln\Big(1-\beta\lambda_1\Big)}.$$
In order to insure that $r\in \mathbb{N}$ the chosen regularization parameter is given by
\begin{equation}
r=\text{max}\left\{\ \left \lceil{\frac{\ln\Big(\frac{\delta}{\beta\sqrt{\lambda_1}}\Big)}{\ln\Big(1-\beta\lambda_1\Big)}}\right \rceil, 1\right\}.    \label{rdef}
 \end{equation} 
From here we have a method to pick the parameter $r \in\N$ with respect to the known noise level. In our numerical experiments, we noticed that this choice of $r\leq5$.

 Before proceeding with the numerical examples, we address the stability of the imaging function $W_{\text{LDSM}}(z)$ given by \eqref{DSMdef} with respect to a given/measured perturbed far-field operator
 It is well known that if  
 $$\|F^{\delta}-F\|<\delta \quad \text{ we have that} \quad  \|F_{\#}^{\delta}-F_{\#}\|<C(1+|\ln(\delta)|)\delta$$ 
 for some $C>0$ independent of $\delta$ see for e.g. \cite{RegFM-AL}. We present a lemma that will address an important property before showing the stability result.  
 \begin{lemma}
 Assume that $P_{r,\epsilon}(F_{\#}): L^2({\mathbb{S}^{d-1}}) \to L^2({\mathbb{S}^{d-1}})$ is defined as above in \eqref{polyF}, then we have that 
 \begin{align*}
     \|P_{r,\epsilon}(F^{\delta}_{\#})-P_{r,\epsilon}(F^{}_{\#})\|\leq C\sum_{m=0}^{p-1}\left\|(F^{\delta}_{\#})^{m}-(F_{\#})^m\right\|\leq C(1+\ln(\delta))\delta.
 \end{align*}
 \label{polyfact}
 \end{lemma}
 \begin{proof}
     To begin the argument, we make the observation that we can always factorize terms of the form $$(F^{\delta}_{\#})^p-(F_{\#})^p=(F^{\delta}_{\#}-F_{\#})\sum_{m=0}^{p-1}(F_{\#}^{\delta})^m(F_{\#})^{p-1-m}$$
     where we define $Q_{p-1}(F^{\delta}_{\#},F_{\#})=\sum_{m=0}^{p-1}(F_{\#}^{\delta})^m(F_{\#})^{p-1-m}$ which is a polynomial of two variables and has degree $p-1.$ We focus our attention to the following term 
     \begin{align*}
         \| (F^{\delta}_{\#})^p-(F_{\#})^p\|&=\| Q_{p-1}(F^{\delta}_{\#},F_{\#})(F^{\delta}_{\#}-F_{\#})\|\\
         &\leq \| Q_{p-1}(F^{\delta}_{\#},F_{\#})\| (1+|\ln(\delta)|)\delta\\
         & \leq C\| F_{\#}\|^{p-1}(1+|\ln(\delta)|)\delta
     \end{align*}
     where $\| F_{\#}\|^{p-1}$ is bounded. Thus, we have $\| (F^{\delta}_{\#})^p-(F_{\#})^p\|\leq C(1+|\ln(\delta)|)\delta$ where $C$ is a constant independent of $\delta$.
 \end{proof}

With this result we are now able to prove stability of the imaging function $W_{\text{LDSM}}(z)$ defined in \eqref{DSMdef}. Here was assume that only the perturbed operator $F^\delta$ is known and we prove that the imaging function using the perturbed operator is uniformly close to the imaging function using the unperturbed operator. 

\begin{theorem}
    Assume that $\|F^{\delta}-F\|<\delta$ as $\delta\longrightarrow 0$, then 
    $$\Big|\|P_{r,\epsilon}(F^{\delta}_\#)\phi_z\|^2_{L^2({\mathbb{S}^{d-1}})}- \|P_{r,\epsilon}(F^{}_\#)\phi_z\|^2_{L^2({\mathbb{S}^{d-1}})}\Big|\leq C(1+|\ln(\delta)|)\delta  \quad \text{as }\quad \delta\longrightarrow 0$$
    uniformly on compact subsets of $\R^d$. 
\end{theorem}
\begin{proof}
  Using the $L^2(\mathbb{S}^{d-1})$ norm and its inner product we have the following inequalities 
\begin{align*}
&\Big|\|P_{r,\epsilon}(F^{\delta}_\#)\phi_z\|^2_{L^2({\mathbb{S}^{d-1}})}-\|P_{r,\epsilon}(F^{}_\#)\phi_z\|^2_{L^2({\mathbb{S}^{d-1}})}\Big|\\
&\leq  \Big| \| P_{r,\epsilon}(F_{\#}^{\delta})\phi_z \|_{L^2({\mathbb{S}^{d-1}})} \|(P_{r,\epsilon}(F_{\#}^{\delta})-P_{r,\epsilon}(F_{\#}))\phi_z\|_{L^2({\mathbb{S}^{d-1}})}\\
&\hspace{5cm}+\| P_{r,\epsilon}(F_{\#}^{\delta})-P_{r,\epsilon}(F_{\#}))\phi_z\|_{L^2({\mathbb{S}^{d-1}})}\| P_{r,\epsilon}(F^{}_{\#})\phi_z\|_{L^2({\mathbb{S}^{d-1}})}\Big|\\
&\leq \Big(\| P_{r,\epsilon}(F_{\#}^{\delta})\| \| P_{r,\epsilon}(F_{\#}^{\delta})-P_{r,\epsilon}(F_{\#})\|+\|P_{r,\epsilon}(F_{\#}^{\delta})-P_{r,\epsilon}(F_{\#})\| \| P_{r,\epsilon}(F^{}_{\#})\|\Big)\|\phi_z\|^2_{L^2({\mathbb{S}^{d-1}})}
\end{align*}
where on the second line we have added and subtracted terms and used the Cauchy–Schwarz inequality.  It is clear from \eqref{polyF} and Lemma \ref{polyfact} that $\| P_{r,\epsilon}(F_{\#}^{\delta})\|$ and $\| P_{r,\epsilon}(F^{}_{\#})\|$ are both bounded with respect to $\delta \in (0,1)$. Thus we have that 
$$\Big|\|P_{r,\epsilon}(F^{\delta}_\#)\phi_z\|^2_{L^2({\mathbb{S}^{d-1}})}-\|P_{r,\epsilon}(F^{}_\#)\phi_z\|^2_{L^2({\mathbb{S}^{d-1}})}\Big|\leq C  \| P_{r,\epsilon}(F_{\#}^{\delta})-P_{r,\epsilon}(F^{}_{\#})\|.$$
 Using Lemma \eqref{polyfact} we have  $\| P_{r,\epsilon}(F_{\#}^{\delta})-P_{r,\epsilon}(F^{}_{\#})\|\leq C(1+|\ln(\delta)|)\delta$ as $\delta\longrightarrow 0$. This last inequality is the final item to show the desired stability. Thus we have 
 $$\Big|||P_{r,\epsilon}(F^{\delta}_\#)\phi_z||^2_{L^2({\mathbb{S}^{d-1}})}-||P_{r,\epsilon}(F_\#)\phi_z||^2_{L^2({\mathbb{S}^{d-1}})}\Big|\leq C(1+|\ln(\delta)|)\delta \quad \text{as }\quad \delta\longrightarrow 0$$
 proving the claim.
\end{proof}
The stability result closes up the analysis about the Landweber direct sampling method connecting this direct sampling method and factorization method. In the following section we present numerical results using the imaging function to recover multiple types of scatterers. 

\section{Numerical Validation}
\subsection{Boundary Integral Equations}
We first derive the boundary integral equation to compute far-field data for arbitrary domains in two dimensions which are defined through a smooth parametrization. Note that the derivation is also valid in three dimensions by changing the corresponding fundamental solution in the integral operators. 

Recall, that the given scatterer $D$ is illuminated by an incident plane wave of the form $u^i=\text{e}^{\text{i} k x \cdot \hat{y} }$ with incident direction $\hat{y}\in \mathbb{S}^{1}$ (the unit circle), then the direct scattering problem is given by: find the total field $u \in H^1(D)$ and scattered field $u^s \in H^1_{loc}(\R^d \setminus \overline{D})$ such that  
\begin{align}
\Delta u^s +k^2 u^s=0  \quad \textrm{ in }  \R^2 \setminus \overline{D}  \quad  \text{and} \quad  \Delta u +k^2 nu=0  \quad &\textrm{ in } \,  {D}  \label{direct1new}\\
 (u^s+u^i )^+ - u^-=0  \quad  \text{and} \quad \partial_\nu (u^s+u^i )^+ + \eta (u^s+u^i )^+=  {\partial_{\nu} u^-} \quad &\textrm{ on } \,  \partial D \label{direct3new}\\
 \lim\limits_{r \rightarrow \infty} r^{1/2} \left( {\partial_r u^s} -\text{i} k u^s \right)=0 \label{SRCnew}&\,.
\end{align}
We use a single-layer ansatz to derive a $2\times 2$ system of boundary integral equations. Precisely, we take
\begin{eqnarray}
 u^s(x)=\mathrm{SL}_{k}\varphi(x)\,,\quad x\in\R^2 \setminus \overline{D}\qquad\text{and}\qquad u(x)=\mathrm{SL}_{k\sqrt{n}}\psi(x)\,,\qquad x\in D\,,
 \label{start}
\end{eqnarray}
where 
\[\mathrm{SL}_{k}\phi(x)=\int_{\partial D}\Phi_k(x,y)\phi(y)\,\mathrm{d}s\,,\qquad  x\notin \partial D\]
with $\Phi_k(x,y)$ the fundamental solution of the Helmholtz equation in two dimensions. Here, $\varphi$ and $\psi$ are yet unknown functions on $\partial D$. On the boundary, we have 
\[u^s(x)=\mathrm{S}_{k}\varphi(x)\qquad\text{and}\qquad  u(x)=\mathrm{S}_{k\sqrt{n}}\psi(x)\,,\]
where
\[\mathrm{S}_{k}\phi(x)=\int_{\partial D}\Phi_k(x,y)\phi(y)\,\mathrm{d}s\,,\qquad x\in \partial D\,.\]
Because of $u-u^s=u^i$, we obtain the first boundary integral equation
\begin{eqnarray}
 \mathrm{S}_{k\sqrt{n}}\psi-\mathrm{S}_{k}\phi=u^i\,.
 \label{first}
\end{eqnarray}
Taking the normal derivative of (\ref{start}) and the jump conditions yields on the boundary
\[\partial_\nu u^s(x)=\left(-\frac{1}{2}\mathrm{I}+\mathrm{K}'_{k}\right)\varphi(x)\qquad\text{and}\qquad  \partial_\nu u(x)=\left(\frac{1}{2}\mathrm{I}+\mathrm{K}'_{k\sqrt{n}}\right)\psi(x)\,,\]
where 
\[\mathrm{K}'_{k}\phi(x)=\int_{\partial D}\partial_{\nu(x)}\Phi_k(x,y)\phi(y)\,\mathrm{d}s\,,\qquad x\in \partial D\,.\]
Because of $\partial_\nu u-\partial_\nu u^s-\eta u^s=\partial_\nu u^i+\eta u^i$, we obtain the second boundary integral equation
\begin{eqnarray}
\left(\frac{1}{2}\mathrm{I}+\mathrm{K}'_{k\sqrt{n}}\right)\psi-\left(-\frac{1}{2}\mathrm{I}+\mathrm{K}'_{k}\right)\varphi-\eta \mathrm{S}_{k}\varphi=\partial_\nu u^i+\eta u^i\,. 
\label{second}
\end{eqnarray}
After we solve (\ref{first}) and  (\ref{second}) for $\psi$ and $\varphi$, we obtain the far-field by computing
\[u^\infty(\hat{x})=\mathrm{S}_k^\infty \varphi (\hat{x})\,,\]
where 
\begin{eqnarray}
\mathrm{S}^\infty_k\phi(\hat{x})=\int_{\partial D} \mathrm{e}^{-\mathrm{i}k \hat{x}\cdotp y}\phi(y)\,\mathrm{d}s(y)\,,\quad \hat{x}\in \mathbb{S}^1\,.
\label{farfieldexpression}
\end{eqnarray}

The system of boundary integral equations (\ref{first}) and (\ref{second}) is numerically solved with the boundary element collocation method (refer also to \cite{hotspot} for more details). Likewise, the expression (\ref{farfieldexpression}) is approximated.

To test that our solver produces correct results, we derive the corresponding far-field pattern for a disk with radius $R>0$.
The Jacobi-Anger expansion for the incident wave $u^i(x)=\mathrm{e}^{\mathrm{i}kx\cdotp \hat{y} }$ with incident direction $\hat{y} $ is given by
\begin{eqnarray*}
 \mathrm{e}^{\mathrm{i}kx\cdotp \hat{y} }=\sum_{p=-\infty}^\infty \mathrm{i}^p J_p(k|x|) \mathrm{e}^{\mathrm{i}p (\theta -\phi)}\,,
\end{eqnarray*}
where $\theta$ is the polar angle for $x$ and $\phi$ is the polar angle for $\hat{y} $. The scattered field in the exterior is given by
\begin{eqnarray*}
 u^s(r\hat{x})=\sum_{p=-\infty}^\infty \mathrm{i}^p a_pH_p^{(1)}(kr)\mathrm{e}^{\mathrm{i}p(\theta -\phi)}\,,\qquad r>R\,,
\end{eqnarray*}
where $\hat{x}=x/r\in \mathbb{S}^1$. The field inside of $D$ is given by
\begin{eqnarray*}
 u(r\hat{x})=\sum_{p=-\infty}^\infty \mathrm{i}^p b_pJ_p(k\sqrt{n}r)\mathrm{e}^{\mathrm{i}p(\theta -\phi)}\,,\qquad r<R\,.
\end{eqnarray*}
The first boundary condition $u^s-u=-u^i$ yields
\begin{eqnarray}
 H_p^{(1)}(kR)a_p-J_p(k\sqrt{n}R)b_n=-J_p(kR)\,.
 \label{line1}
\end{eqnarray}
The second boundary condition $\partial_\nu u^s+\eta u^s- \partial_\nu u=-\partial u^i-\eta u^i$ gives
\begin{eqnarray}
 kH_p^{(1)'}(kR)a_p+\eta H_p^{(1)}(kR)a_p- k\sqrt{n}J_p'(k\sqrt{n}R)b_p=-kJ_p'(kR)-\eta J_p(kR)\,.
 \label{line2}
\end{eqnarray}
Equations (\ref{line1}) and (\ref{line2}) can be written as
\begin{align*}
 \left(
 \begin{array}{cc}
  H_p^{(1)}(kR) & -J_p(k\sqrt{n}R)\\
  kH_p^{(1)'}(kR)+\eta H_p^{(1)}(kR) & - k\sqrt{n}J_p'(k\sqrt{n}R)
 \end{array}
 \right)\left(
 \begin{array}{c}
  a_p\\
b_p
  \end{array}
 \right)= \left(
 \begin{array}{c}
  -J_p(kR)\\
  -kJ_p'(kR)-\eta J_p(kR)
 \end{array}
 \right)\,.
\end{align*}
The solution $a_p$ (using Cramer's rule) is given by
\begin{eqnarray*}
 a_p=-\frac{
  k\sqrt{n}J_p(kR)J_p' (k\sqrt{n}R)-J_p(k\sqrt{n}R)\left(kJ_p'(kR)+\eta J_p(kR)\right)
 }{
  k\sqrt{n}H_p^{(1)}(kR)J_p'(k\sqrt{n}R)-J_p(k\sqrt{n}R)\left(kH_p^{(1)'}(kR)+\eta H_p^{(1)}(kR)\right)
 }\,.
\end{eqnarray*}
The far-field is expressed by
\begin{eqnarray}
 u^\infty(\hat{x}, \hat{y})=\frac{4}{\mathrm{i}}\sum_{p=-\infty}^\infty a_p \mathrm{e}^{\mathrm{i}p(\theta -\phi)}\,.
 \label{farfieldseries}
\end{eqnarray}
Let ${\bf F}_k\in \mathbb{C}^{64 \times 64}$ be the matrix containing the far-field data for $64$ equidistant incident directions and $64$ evaluation points for the disk with radius $R$ with parameters, $\eta$, $n$ and given wave number $k$ obtained by (\ref{farfieldseries}). We denote by ${\bf F}_k^{(N_f)}$ the far-field data obtained through the boundary element collocation method, where $N_f$ denotes the number of faces in the method. Note that the number of collocation nodes is $3\cdotp N_f$. The absolute error is defined by 
$$\varepsilon_k^{(N_f)} :=\|{\bf F}_k-{\bf F}_k^{(N_f)}\|_2 .$$
In Table \ref{tablefarfield}, we show the absolute error of the far-field for 64 incident directions and 64 evaluation point, for a disk with radius $R=1$ and the parameters $\eta=2+\mathrm{i}$, and $n=4+\mathrm{i}$ and the wave numbers $k=2$, $k=4$, and $k=6$. As we can observe, we obtain very accurate results using $120$ collocation nodes.

\begin{table}[!ht]
\centering
 \begin{tabular}{c|c|c|c|}
  $N_f$ & $\varepsilon_2^{(N_f)}$ & $\varepsilon_4^{(N_f)}$ & $\varepsilon_6^{(N_f)}$ \\
  \hline 
  10 & 0.82745 & 9.75548 & 74.46130\\
20 & 0.01051 & 0.41988 & 3.07890\\
40 & 0.00089 & 0.00556 & 0.03872\\
80 & 0.00011 & 0.00018 & 0.00108\\
\hline
 \end{tabular}
 \caption{\label{tablefarfield}Absolute error of the far-field with 64 equidistant incident directions and 64 evaluation for the disk with $R=1$ and the parameters, $\eta=2+\mathrm{i}$, and $n=4+\mathrm{i}$ for varying number of faces (collocation nodes). The wave numbers are $k=2$, $k=4$, and $k=6$.}
\end{table}

\subsection{Numerical Examples}
For the numerical examples we will be using the discretized form of the operator $\textbf{F}_{\#}$ which we can get from the discretized far-field operator $\textbf{F}$ i.e. 
$$\textbf{F}=\Big[u^\infty(\hat{x}_i,\hat{y}_j)\Big]^{64}_{i,j=1}.$$ We can discretize such that   
$$\hat{x}_i=\hat{y}_i=(\cos(\theta_i),\sin(\theta_i))\quad \text{ where } \quad \theta_i=2\pi(i-1)/64 \,\, \text{ for} \,\,i=1,\dots,64.$$ 
We get then $\textbf{F}$ which is a $64 \times 64$ complex valued matrix with $64$ incident and observation directions. 
An additional component needed is the vector $\phi_z$ which we compute by
$$\boldsymbol{\phi}_z=\big(\text{e}^{-\text{i}k\hat{x}_1\cdot z},\dots, \text{e}^{-\text{i}k\hat{x}_{64} \cdot z}\big)^\top\hspace{.5cm}\text{where}\hspace{.5cm}z\in \mathbb{R}^2.$$ 
In order to model experimental error in the data we add random noise to the discretized far-field operator $\mathbf{F}$ such that
$$\mathbf{F}^{\delta}=\Big[ \mathbf{F}_{i,j}(1+\delta\mathbf{E}_{i,j})\Big]^{64}_{i,j=1}\hspace{.5cm}\text{where}\hspace{.5cm} \| \mathbf{E}\|_2=1.$$
Here, the matrix $\mathbf{E} \in \C^{64 \times 64}$ is taken to have random entries and $0<\delta \ll 1$ is the relative noise level added to the data. This gives that the relative error is given by $\delta$.

Thus, numerically we can approximate the imaging function by 
$$W_{\text{LDSM}}(z)=\left\|P_{r,\epsilon}(\textbf{F}^{\delta}_\#)\boldsymbol{\phi}_z\right\|^4_{L^2(\mathbb{S})}$$
where we use the 4--th power to increase the resolution. We need to numerically be able to compute the approximation polynomial $P_{r,\epsilon}(t)$ in order to continue discretizing the imagining functional. Recall, that the matrix $\textbf{F}^{\delta}_{\#}=|\text{Re}(\textbf{F}^{\delta})|+|\text{Im}(\textbf{F}^{\delta})|$ where we have that 
$$\text{Re}(\textbf{F}^{\delta})=\frac{1}{2} \big[\textbf{F}^{\delta}+(\textbf{F}^{\delta})^{*}\big] \quad \text{and} \quad \text{Im}(\textbf{F}^{\delta})=\frac{1}{2 \text{i}} \big[\textbf{F}^{\delta}-(\textbf{F}^{\delta})^{*}\big].$$ 
Here, the absolute value of the matrices are define via its eigenvalue decomposition. With the computed  $\textbf{F}^{\delta}_{\#}$ we compute the singular values denoted $s_j$ for $j=1 , \cdots , 64$. Thus, in order to use our Landweber direct sampling method, we need the construction of the polynomial 
$P_{r,\epsilon}(t)$ such that for all $t\in[0,s_1]$ approximates the function $\Gamma_{r}(t)$ defined in the previous section. For the experiments we construct the polynomial such that  
$$P_{r,\epsilon}(t)=\sum_{k=1}^Mc_kt^k\hspace{.5cm}\text{such that}\hspace{.5cm}P_{r,\epsilon}(t_{\ell})=\frac{1}{\sqrt{t_{\ell}}} \big(1-(1-\beta t_{\ell})^r\big),$$
where $M$ is the degree of the polynomial and $t_{\ell} \in [0,s_1]$ are the interpolation points. We consider three different interpolation points over the interval $[0,s_1]$; 
\begin{enumerate}
\item $\ell=1,\ldots,100$ and $t_{\ell}$ are equally spaced,
\item $t_{\ell}$ are the singular values of ${\bf F}_{\#}$ i.e. $t_j=s_j$, 
\item $t_{\ell}$ are the 32 Gaussian quadrature points on the interval $[0,s_1]$.
\end{enumerate} 
We compute the regularization parameter $r \in \N$ where it is defined in \eqref{rdef}. In addition, one uses a spectral cut-off to compute the coefficients $c_k$ where the cut-off parameter is fixed to be $10^{-8}$ in all the numerical examples. 

Once the approximating polynomial is computed, we can numerically approximate the new imaging function $W_{\text{LDSM}}(z)$. We will discuss the construction of the polynomial $P_{r,\epsilon}(t)$ in terms of the degree and how the choice affects the numerical examples.  As mentioned in Theorem \eqref{distD}, we will see in the examples the decay as the sampling point $z$ moves away from the scatterer/boundary using the approximation polynomial applied to the solution operator. We consider the following three domains: disk with radius one, a kite, and a peanut. Their respective parameterizations are given by 
$$\partial D=\big(\cos(\theta),\sin(\theta)\big)^\top, \quad \partial D=\big(-1.5\sin(\theta),\cos(\theta)+0.65\cos(2\theta)-0.65\big)^\top$$
and 
$$\partial D=2\sqrt{\frac{\sin(\theta)^2}{2}+\frac{\cos(\theta)^2}{10}}\big(\cos(\theta),\sin(\theta)\big)^\top.$$
Note, that for the kite and peanut shaped scatterers, the far-field data was computed as described in the previous section using $N_f=128$.
In all of our examples, we address the different ways of interpolating the polynomial $P_{r,\epsilon}(t)$ on the interval $[0,s_1]$ and use its construction and representation to approximate the solution operator. \\
   
\noindent\textbf{Example 1. Recovering a peanut region:}\\
\noindent For the peanut shaped domain, we assume that the refractive index is $n=4+\text{i}$ and boundary parameter $\eta=2+\text{i}.$ Here, we will take $k=2\pi$ as the wave number and we let $\delta=0.10$ which corresponds to the $10\%$ random noise added to the data. In this first example, we address the construction of the polynomial $P_{r,\epsilon}(t)$ with respect to the degree. The first image has an interpolating polynomial of degree $M=4$ and on the second image the degree of the polynomial is 6. 
\begin{figure}[ht]
\centering 
\includegraphics[width=15cm]{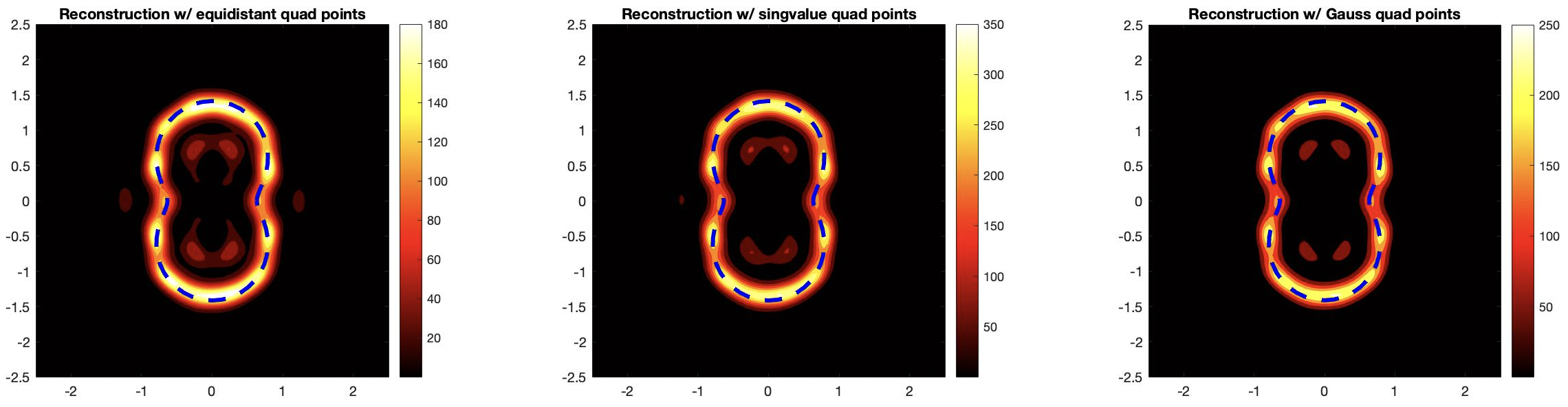}
\caption{Reconstruction using an interpolating polynomial of degree $M=4$ of peanut region by the Landweber direct sampling method. Images left to right:  reconstruction using equidistant points, singular values, and Gaussian quadrature points. }
\label{Peanut4}
\end{figure}

\begin{figure}[ht]
\centering 
\includegraphics[width=15cm]{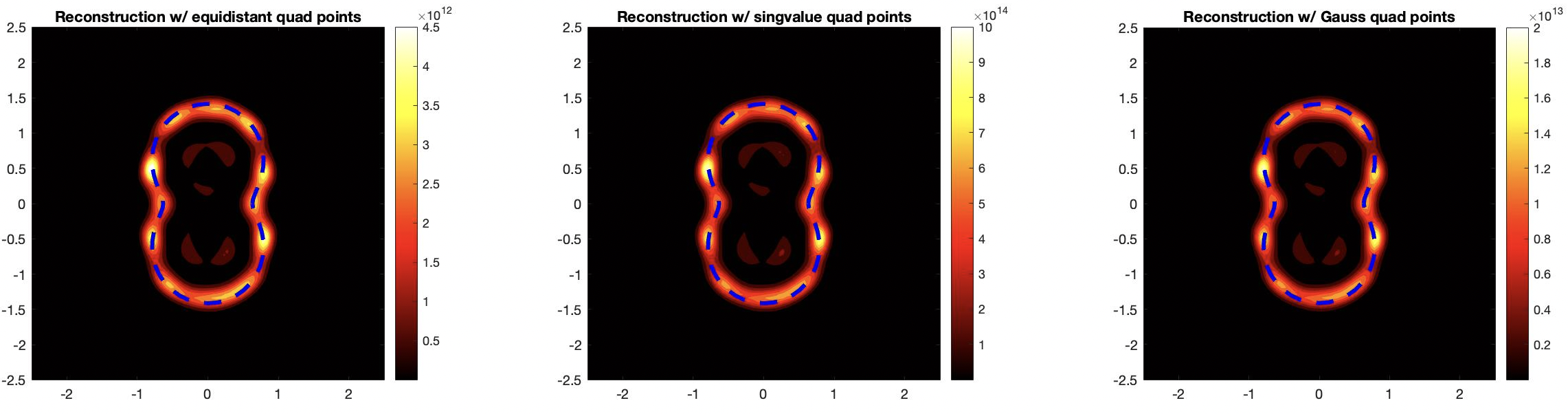}
\caption{Reconstruction using an interpolating polynomial of degree $M=6$ of peanut region by the Landweber direct sampling method. Images left to right:  reconstruction using equidistant points, singular values, and Gaussian quadrature points. }
\label{Peanut6}
\end{figure}

In Figure \eqref{Peanut4} and \eqref{Peanut6}, we see that both images are very similar and both give a good approximation of the scatterer. We tried many  degrees for the interpolating polynomial but we chose to present degree 4 and 6. With any degree, the only change we see is that the values at the boundary are higher. We can conclude that using any degree for the interpolating polynomial will be sufficient and enough to approximate the solution operator. Thus, without loss of generality for the rest of the numerical examples we assume that the degree of the polynomial can be taken to be $M=4$.\\

\noindent\textbf{Example 2. Recovering a peanut region with $20\%$ noise:}\\
\noindent For this reconstruction, we take the same values for the physical parameters as example 1. The difference here is that we fix the degree of the interpolating polynomial to be $M=4$, we do all the interpolating methods, and lastly we add $20\%$ random noise to the data. \\

\begin{figure}[ht]
\centering 
\includegraphics[width=15cm]{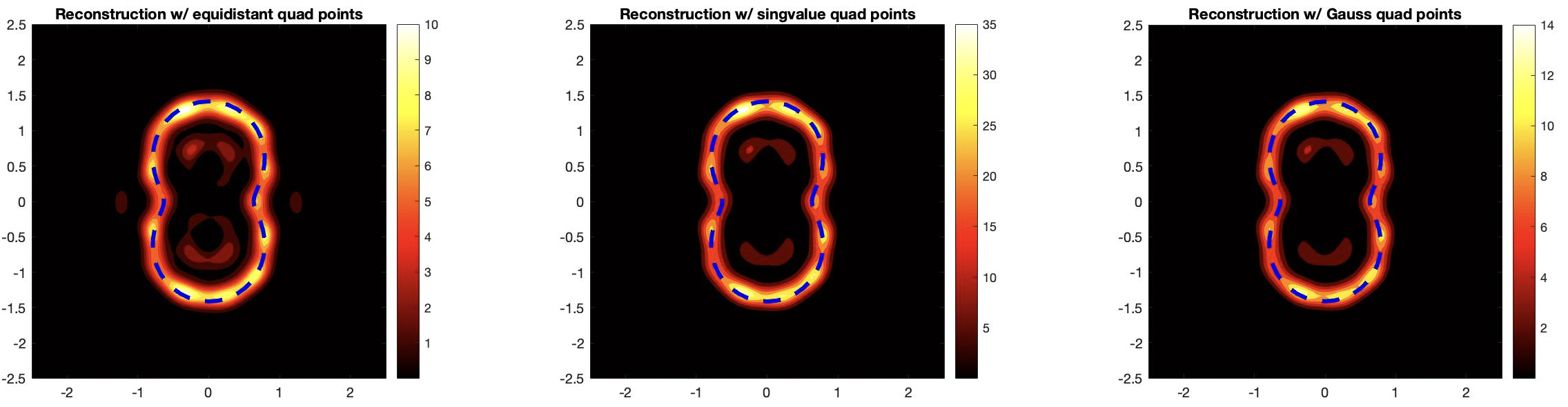}
\caption{Reconstruction using an interpolating polynomial of degree $M=4$ of peanut region by the Landweber direct sampling method with $20\%$ noise. Images left to right:  reconstruction using equidistant points, singular values, and Gaussian quadrature points. }
\label{peanut20per}
\end{figure}

In Figure \eqref{peanut20per}, we see that with even more random noise added, the reconstruction only changes with respect to the values at the boundary in comparison to Figure \eqref{Peanut4}. \\

\noindent\textbf{Example 3. Recovering a kite region:}\\
For this numerical experiment, we have fixed the degree of the approximation polynomial to be $M=4$, the refractive index to be $n=4+\text{i},$ and boundary parameter $\eta=2+\text{i}.$ Here, we will take $k=6$ as the wave number and $\delta=0.10$ which corresponds to the $10\%$ random noise added to the data. 

\begin{figure}[ht]
\centering 
\includegraphics[width=15cm]{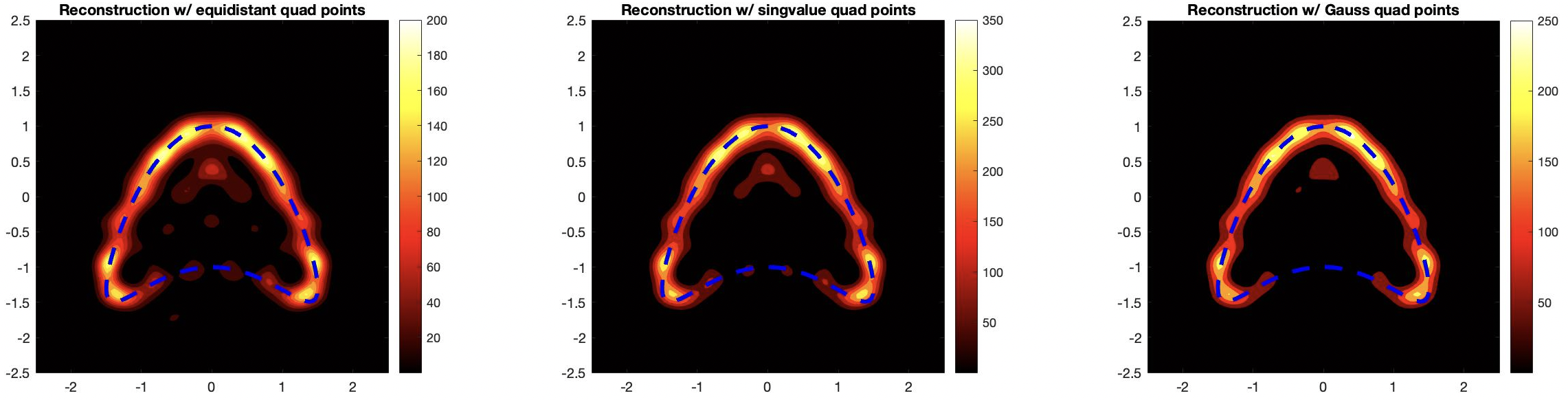}
\caption{Reconstruction using an interpolating polynomial of degree $M=4$ of kite scatterer by the Landweber direct sampling method with $10\%$ noise. Images left to right:  reconstruction using equidistant points, singular values, and Gaussian quadrature points.}
\label{kite10}
\end{figure}

In the next example, we compare \eqref{kite10} with the same reconstruction but using a noise level of $20\%$ and the wave number $k=2\pi.$

\begin{figure}[ht]
\centering 
\includegraphics[width=15cm]{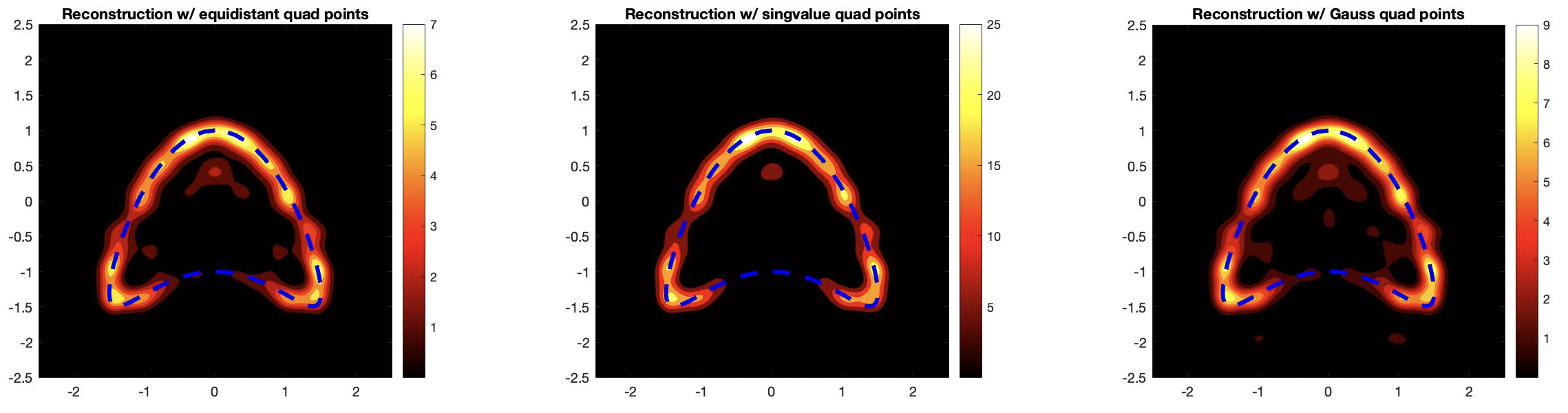}
\caption{Reconstruction using an interpolating polynomial of degree $M=4$ of kite scatterer by the Landweber direct sampling method with $20\%$ noise. Images left to right:  reconstruction using equidistant points, singular values, and Gaussian quadrature points.}
\label{kite20}
\end{figure}

In Figure \eqref{kite10} and \eqref{kite20}, both reconstructions are very similar. The change is based on the values at the boundary and how big they are. However, even with different noise levels we still capture most of the scatterers. For the last two reconstructions, we will analyze the unit circle and address a change of physical parameters to see how our indicator function performs when we modify these. \\

\noindent\textbf{Example 4. Recovering a circle region:}\\
For this numerical experiment, we have fixed the degree of the approximation polynomial to be $M=4$, the refractive index to be $n=3,$ and boundary parameter $\eta=6+4\text{i}.$ Here, we will take $k=4$ as the wave number and $\delta=0.15$ which corresponds to the $15\%$ random noise added to the data. 
\begin{figure}[ht]
\centering 
\includegraphics[width=15cm]{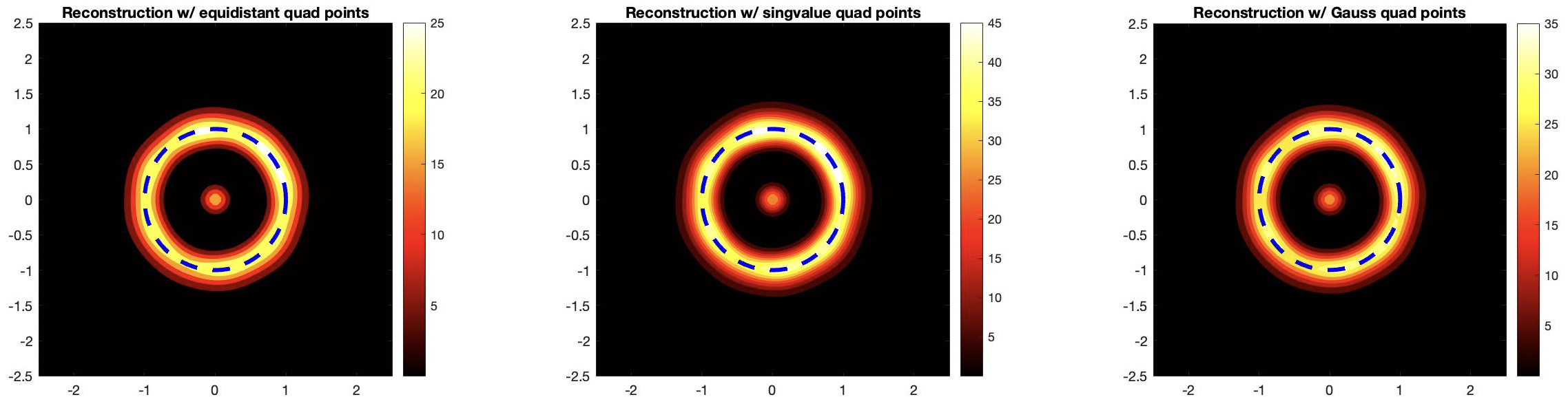}
\caption{Reconstruction using an interpolating polynomial of degree $M=4$ of circle scatterer by the Landweber iteration method with $15\%$ noise. Images left to right:  reconstruction using equidistant points, singular values, and Gaussian quadrature points.}
\label{circle1}
\end{figure}

For this last example, we change the physical parameters to be $n=5$ and $\eta=2.5+\text{i}$ and we keep the wave number and noise the same. 
\begin{figure}[ht]
\centering 
\includegraphics[width=15cm]{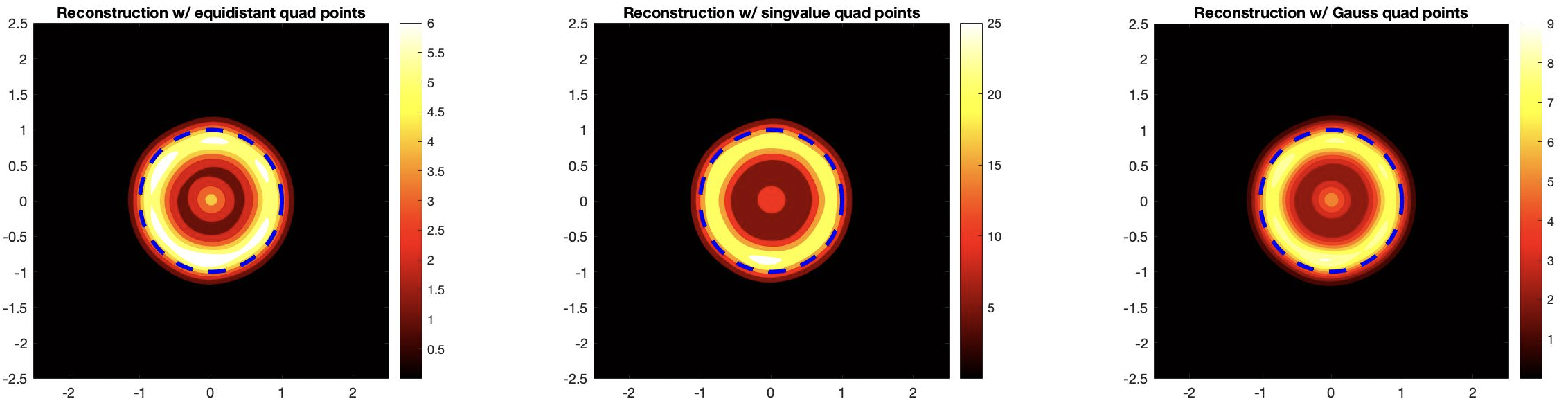}
\caption{Reconstruction using an interpolating polynomial of degree $M=4$ of circle scatterer by the Landweber iteration method with $15\%$ noise. Images left to right:  reconstruction using equidistant points, singular values, and Gaussian quadrature points.}
\label{circle2}
\end{figure}

We see that in both images, \eqref{circle1} and \eqref{circle2} the location of the scatterers are known. Although, changing the physical parameters gives us a better reconstruction in the second image, we still can fully reconstruct the boundary of the scatterer in the first example for the circle. In conclusion, our indicator function does perform well in terms of determining the location, the shape, and the size of the scatterer when varying either the noise level, the physical parameters, or the shape of the scatter. \\

\section{Conclusion}
In this study, we investigated a novel direct sampling method linked to the factorization method. This generalizes the work in \cite{harris-dsm} to the case when the scatterer has complex-valued coefficients i.e. $F$ may not be a diagonalizable operator. To achieve this, we developed a factorization of the far-field operator and then analyzed the operator to derive the new imaging function. We have derived the resolution analysis as well as the stability of the proposed reconstruction algorithm. A further extension to the work in \cite{harris-dsm} is the discrepancy principle used to determine the regularization parameter given in equation \eqref{rdef}. Also, a detailed numerical study is presented to show the stability and accuracy of the method. There are further questions to be explored for this scattering problem, such as: does the far-field data uniquely determine the coefficients as well as studying direct sampling methods for the case with two boundary parameters(see for e.g. \cite{fmconductbc,te-2cbc}).\\

\noindent{\bf Acknowledgments:} The research of R. Ceja Ayala and I. Harris is partially supported by the NSF DMS Grant 2107891.


\end{document}